\documentclass[pdflatex,sn-mathphys-num]{sn-jnl}

\usepackage{graphicx}%
\usepackage{multirow}%
\usepackage{amsmath,amssymb,amsfonts}%
\usepackage{amsthm}%
\usepackage{mathrsfs}%
\usepackage[title]{appendix}%
\usepackage{xcolor}%
\usepackage{textcomp}%
\usepackage{manyfoot}%
\usepackage{booktabs}%
\usepackage{algorithm}%
\usepackage{algorithmicx}%
\usepackage{algpseudocode}%
\usepackage{listings}%
\usepackage[utf8]{inputenc}

\theoremstyle{thmstyleone}%
\newtheorem{theorem}{Theorem}
\theoremstyle{thmstyletwo}%
\newtheorem{Example}{Example}%
\newtheorem{Remark}{Remark}%
\newtheorem{Lemma}{Lemma}
\newtheorem{Corollary}{Corollary}
\newtheorem{Proposition}{Proposition}

\theoremstyle{thmstylethree}%
\newtheorem{Definition}{Definition}%

\begin{document}

\title[On some generalized geometric constants with two parameters in Banach spaces]{On some generalized geometric constants with two parameters in Banach spaces}


\author[1]{\fnm{Yuxin} \sur{Wang}}\email{y24060028@stu.aqnu.edu.cn}
\equalcont{These authors contributed equally to this work.}

\author*[1,2]{\fnm{Qi} \sur{Liu}}\email{liuq67@aqnu.edu.cn}
\equalcont{These authors contributed equally to this work.}

\author[1]{\fnm{Haoyu} \sur{Zhou}}\email{061024028@stu.aqnu.edu.cn}
\equalcont{These authors contributed equally to this work.}

\author[1]{\fnm{Jinyu} \sur{Xia}}\email{y23060036@stu.aqun.edu.cn}
\equalcont{These authors contributed equally to this work.}

	\author[3]{\fnm{Muhammad} \sur{Toseef}}\email{toseefrana95@email.com}
\equalcont{These authors contributed equally to this work.}

\affil*[1]{\orgdiv{School of Mathematics and Physics}, \orgname{Anqing Normal University}, \orgaddress{\city{Anqing}, \postcode{246133}, \country{China}}}

\affil[2]{\orgdiv{International Joint Research Center of Simulation and Control for
		Population Ecology of Yantze River in Anhui}, \orgname{Anqing Normal University}, \orgaddress{\city{Anqing}, \postcode{246133},  \country{China}}}

\affil[3]{\orgdiv{Ministry of Education Key Laboratory for NSLSCS, School of Mathematical Sciences}, \orgname{Nanjing Normal University}, \orgaddress{\city{Nanjing}, \postcode{210023},  \country{China}}}


\abstract{	In this paper, we build upon the $\mathcal{T_X}$ constant that was introduced by Alonso and Llorens-Fuster in 2008. Through the incorporation of suitable parameters, we have successfully generalized the aforementioned constant into two novel forms of geometric constants, which are denoted as $\mathcal{T}_1(\lambda,\mu,\mathcal{X})$ and $\mathcal{T}_2(\kappa,\tau,\mathcal{X})$. First, we obtained some basic properties of these two constants, such as the upper and lower bounds. Next, these two constants served as the basis for our characterization of Hilbert spaces.  More significantly, our findings reveal that these two constants exhibit a profound and intricate interrelation with other well-known constants in Banach spaces. Finally, we characterized uniformly non-square spaces by means of these two constants.  }




\maketitle

\section{Introduction and preliminaries}
In the present article, the following assumptions are made: suppose \(\mathcal{X}\) is a real Banach space satisfying \(\operatorname{dim}(\mathcal{X})\geq2\). The notations \(\mathcal{B(X)}\) and \(\mathcal{S(X)}\) are employed to denote the unit ball and the unit sphere of the Banach space \(\mathcal{X}\), and $\wedge$ and $\vee$ denote the minimum and maximum values of the two, respectively.   

The geometric theory of Banach spaces relies heavily on geometric constants, which are indispensable for the analysis and comprehension of space properties. In the contemporary academic landscape, a plethora of studies have emerged, delving into diverse geometric constants of the Banach space $\mathcal{X}$.

As per the definition in \cite{11},  a Banach space \(\mathcal{X}\) is classified as uniformly non-square when there exists a \(\delta\in(0,1)\) with the following property: for all \(x_1,x_2\in\mathcal{S(X)}\),  the inequality \(\left\|\frac{x_1 + x_2}{2}\right\|\wedge\left\|\frac{x_1 - x_2}{2}\right\|\leq 1-\delta\) holds.  The constant $$\mathcal{J(X)}=\sup\left\{\|x_1 + x_2\|\wedge\|x_1 - x_2\|:x_1,x_2\in \mathcal{S(X)}\right\}$$ is known as the non-square or James constant \cite{15} of \(\mathcal{X}\). 

In Clarkson's work \cite{13}, he introduced the von Neumann-Jordan constant $\mathcal{C_{NJ}(X)}$. Through ingenious argumentation, he indirectly transformed the parallelogram law into the geometric constant form: 
$$\mathcal{C_{NJ}(X)}=\sup\left\{\frac{\|x_1 + x_2\|^2+\|x_1 - x_2\|^2}{2\|x_1\|^2 + 2\|x_2\|^2}:x_1,x_2\in\mathcal{X},\|x_1\|^2+\|x_2\|^2\neq0\right\}.$$ 

Later, in 2008, Alonso et al. \cite{14} restricted the aforementioned $\mathcal{C_{NJ}(X)}$  constant to the unit sphere and defined the following $\mathcal{C'_{NJ}(X)}$ constant:
$$\mathcal{C'_{NJ}(X)}=\sup\left\{\frac{\|x_1+x_2\|^2+\|x_1-x_2\|^2}{4}:x_1,x_2\in\mathcal{S(X)}\right\}.$$

The authors \cite{16} considers the triangles with vertices \(x_1\), \(-x_1\) and \(x_2\), where \(x_1, x_2\in\mathcal{S}(X)\). By making use of the arithmetic means of the variable lengths of the sides of these triangles, the following constants for Banach spaces are defined:  
$$\mathcal{A}_2(\mathcal{X})=\sup_{x_1,x_2\in \mathcal{S}(X)}\frac{\|x_1+x_2\|+\|x_1-x_2\|}{2}.$$

Later, some scholars \cite{19} through the study of the above $\mathcal{A}_2(\mathcal{X})$ constant, have generalized it into the following form: for $\kappa,\tau>0$, $$\mathcal{A}_{\kappa-\tau}(\mathcal{X})=\sup\left\{\frac{\|\kappa x_1+\tau x_2\|+\|\tau x_1-\kappa x_2\|}{2}:x_1,x_2\in \mathcal{S(X)}\right\}.$$

Similarly, the authors \cite{alonso2008geometric}  considers the triangles with vertices \(x_1\), \(-x_1\) and \(x_2\), where \(x_1, x_2\in\mathcal{S}(X)\).  By utilizing the geometric means of the variable lengths of the sides of these triangles, the following constant for Banach spaces is defined:
$$\mathcal{T}(\mathcal{X})=\sup\limits_{x_1,x_2\in \mathcal{S(X)}}\bigg(\|x_1+x_2\|\|x_1-x_2\|\bigg)^\frac{1}{2}.$$
We collect some properties of the  $\mathcal{T}(\mathcal{X})$ constant:

(i) $\sqrt{2}\leq \mathcal{T}(\mathcal{X})\leq2$.

(ii) If $\mathcal{X}$ is a Hilbert space, then $\mathcal{T}(\mathcal{X})=\sqrt{2}$.

(iii) $\mathcal{X}$ is uniformly non-square if and only if $\mathcal{T}(\mathcal{X})<2$.

In this article, through the research and analysis of the $\mathcal{T}(\mathcal{X})$ constant , we have generalized it in two forms and denoted them as $\mathcal{T}_1(\kappa,\tau,\mathcal{X})$ and $\mathcal{T}_2(\kappa,\tau,\mathcal{X})$ respectively. 

In section \ref*{s3}, we mainly introduced the constant $\mathcal{T}_1(\kappa,\tau,\mathcal{X})$. Firstly, we gave the equivalent definition of the constant. Then, we provided its upper and lower bounds, and calculated the value of the constant in the $l_p$ space with $p > 2$. Since $\mathcal{T}_1(\kappa,\tau,\mathcal{X})$ is a generalization of $\mathcal{T}(\mathcal{X})$, we compared the relationship between them. Finally, we characterized the non-uniformly non-square Banach spaces by using the constant $\mathcal{T}_1(\kappa,\tau,\mathcal{X})$. 

In section \ref*{s2}, our research reveals that this particular constant serves as an effective means to describe the properties of Hilbert spaces. Initially, we determined the upper and lower limits of the constant, as well as its specific value within a given Banach space. Subsequently, we explored and analyzed its correlations with several well-known classical constants, including the convexity modulus $\delta_\mathcal{X}(\varepsilon)$, $\mathcal{C'_{NJ}(X)}$, and $\mathcal{T}(\mathcal{X})$. In the final stage of our study, we exploited this constant to define and distinguish uniformly non-square spaces, as well as Banach spaces that exhibit normal structure.

\section{A new constant $\mathcal{T}_1(\kappa,\tau,\mathcal{X})$ }
In this section, we propose the first generalization of the $\mathcal{T}(\mathcal{X})$ constant and denoted the new constant as $\mathcal{T}_1(\kappa,\tau,\mathcal{X})$. Firstly, we demonstrate the form of the $\mathcal{T}_1(\kappa,\tau,\mathcal{X})$ constant: 
\begin{Definition}
	We define the following  constant: for $\kappa,\tau>0$, $$\mathcal{T}_1(\kappa,\tau,\mathcal{X})=\sup\limits_{x_1,x_2\in \mathcal{S(X)}}\bigg(\|\kappa x_1+\tau x_2\|\|\kappa x_1-\tau x_2\|\bigg)^\frac{1}{2}.$$
\end{Definition}
\begin{Remark}
	$\mathcal{T}_1(1,1,\mathcal{X})=\mathcal{T}(\mathcal{X})=\sup\limits_{x_1,x_2\in \mathcal{S(X)}}\bigg(\|x_1+x_2\|\|x_1-x_2\|\bigg)^\frac{1}{2}.$
\end{Remark}
We will commence by presenting an equivalent definition of a constant. Nevertheless, the following crucial lemma is indispensable for our purpose.  

\begin{Lemma}\label{l2}
	Let $\mathcal{X}$ be a Banach space. For any $x_1 \in \mathcal{X}$ and $t \in \mathbb{R}^+$, the function $\psi_{x_1}(t) = \mathop{\sup}\limits_{x_2 \in \mathcal{S(X)}} \|\kappa x_1 +\tau tx_2\|\|\kappa x_1 -\tau tx_2\|$ is an increasing function.
\end{Lemma}
\begin{proof}
	Since $$\sup_{x_2 \in \mathcal{S(X)}} \|\kappa x_1 +\tau tx_2\|\|\kappa x_1 -\tau tx_2\|=\kappa^2\sup_{x_2 \in \mathcal{S(X)}} \|x_1 +\frac{\tau}{\kappa} tx_2\|\|x_1 -\frac{\tau}{\kappa} tx_2\|,$$ then according to a method similar to Lemma 21 in \cite{alonso2008geometric}, we can easily obtain this result.
\end{proof}
\begin{Proposition}
	Let $\mathcal{X}$ be a Banach space. Then,
	
	(i)For any $x_1\in \mathcal{X}$,
	$$\sup\limits_{x_2\in \mathcal{B(X)}}\bigg(\|\kappa x_1+\tau x_2\|\|\kappa x_1-\tau x_2\|\bigg)^\frac{1}{2}=\sup\limits_{x_2\in \mathcal{S(X)}}\bigg(\|\kappa x_1+\tau x_2\|\|\kappa x_1-\tau x_2\|\bigg)^\frac{1}{2}.$$
	
	(ii)$$\begin{aligned}\mathcal{T}_1(\kappa,\tau,\mathcal{X})&=\sup_{x_1\in \mathcal{S(X)}}\sup_{x_2\in\mathcal{B(X)}}\bigg(\|\kappa x_1+\tau x_2\|\|\kappa x_1-\tau x_2\|\bigg)^\frac{1}{2}\\&=\sup_{x_1\in \mathcal{B(X)}}\sup_{x_2\in \mathcal{B(X)}}\bigg(\|\kappa x_1+\tau x_2\|\|\kappa x_1-\tau x_2\|\bigg)^\frac{1}{2}.\end{aligned}$$
\end{Proposition}
\begin{proof}
	(i)Let $x_1\in\mathcal{B(X)}$, it's obviously that \begin{equation}\label{e1}
		\sup\limits_{x_2\in \mathcal{B(X)}}\|\kappa x_1+\tau x_2\|\|\kappa x_1-\tau x_2\|\geq\sup\limits_{x_2\in \mathcal{S(X)}}\|\kappa x_1+\tau x_2\|\|\kappa x_1-\tau x_2\|\geq\sqrt{\kappa^2+\tau^2}
	\end{equation}Now let $\mathcal{A}=\sup\limits_{x_2\in \mathcal{B(X)}}\|\kappa x_1+\tau x_2\|\|\kappa x_1-\tau x_2\|$. We consider the following cases:
	
	Case(i): If $\mathcal{A}=\sqrt{\kappa^2+\tau^2}$, then there is no need to prove.
	
	Case(ii): If $\mathcal{A}>\sqrt{\kappa^2+\tau^2}$, then we take any $\mathcal{B}\in(\sqrt{\kappa^2+\tau^2},A]$, and we take $x_3\in\mathcal{B_X}$ such that $$\|\kappa x_1+\tau x_3\|\|\kappa x_1-\tau x_3\|>\mathcal{B},$$ then by Lemma \ref{l2}, we have $$\bigg\|\kappa x_1+\tau \frac{x_3}{\|x_3\|}\bigg\|\bigg\|\kappa x_1-\tau \frac{x_3}{\|x_3\|}\bigg\|\geq\|\kappa x_1+\tau x_3\|\|\kappa x_1-\tau x_3\|>\mathcal{B},$$thus, we get that \begin{equation}\label{e2}
		\sup\limits_{x_2\in \mathcal{S(X)}}\|\kappa x_1+\tau x_2\|\|\kappa x_1-\tau x_2\|>\mathcal{B}
	\end{equation} is always valid for all $\mathcal{B}\in(\sqrt{\kappa^2+\tau^2},A]$.
	Hence, combine \eqref{e1} and \eqref{e2}, we obtain that 	$$\sup\limits_{x_2\in \mathcal{B(X)}}\|\kappa x_1+\tau x_2\|\|\kappa x_1-\tau x_2\|=\sup\limits_{x_2\in \mathcal{S(X)}}\|\kappa x_1+\tau x_2\|\|\kappa x_1-\tau x_2\|.$$
	Therefore, we completed the proof of (i).
	
	(ii)The first identity follows immediately from (i). Now we prove the second identity. First, it's obviously that \begin{equation}\label{e3}
		\sup_{x_1\in \mathcal{S(X)}}\sup_{x_2\in\mathcal{B(X)}}\bigg(\|\kappa x_1+\tau x_2\|\|\kappa x_1-\tau x_2\|\bigg)^\frac{1}{2}\leq\sup_{x_1\in \mathcal{B(X)}}\sup_{x_2\in \mathcal{B(X)}}\bigg(\|\kappa x_1+\tau x_2\|\|\kappa x_1-\tau x_2\|\bigg)^\frac{1}{2}.
	\end{equation}
	
	Conversely, by Lemma \ref{l2}, if $\|x_1\|\leq 1$, for any fixed $x_2\in\mathcal{B(X)}$, we have $$\bigg\|\kappa\frac{x_1}{\|x_1\|}+\tau x_2\bigg\|\bigg\|\kappa \frac{x_1}{\|x_1\|}-\tau x_2\bigg\|\geq\|\kappa x_1-\tau x_2\|\|\kappa x_2-\tau x_1\|,$$which means that $$\sup\limits_{x_2\in \mathcal{B(X)}}\bigg\|\kappa\frac{x_1}{\|x_1\|}+\tau x_2\bigg\|\bigg\|\kappa \frac{x_1}{\|x_1\|}-\tau x_2\bigg\|\geq\sup\limits_{x_2\in \mathcal{B(X)}}\|\kappa x_1-\tau x_2\|\|\kappa x_2-\tau x_1\|.$$Thus, we get that \begin{equation}\label{e4}
		\sup_{x_1\in \mathcal{S(X)}}\sup_{x_2\in\mathcal{B(X)}}\bigg(\|\kappa x_1+\tau x_2\|\|\kappa x_1-\tau x_2\|\bigg)^\frac{1}{2}\geq\sup_{x_1\in \mathcal{B(X)}}\sup_{x_2\in \mathcal{B(X)}}\bigg(\|\kappa x_1+\tau x_2\|\|\kappa x_1-\tau x_2\|\bigg)^\frac{1}{2}.
	\end{equation}
	Hence, by \eqref{e3}, \eqref{e4} and (i), we obtain that $$\mathcal{T}_1(\kappa,\tau,\mathcal{X})=\sup_{x_1\in \mathcal{B(X)}}\sup_{x_2\in \mathcal{B(X)}}\bigg(\|\kappa x_1+\tau x_2\|\|\kappa x_1-\tau x_2\|\bigg)^\frac{1}{2}.$$
\end{proof}
\begin{Proposition}
	Let $\mathcal{X}$ be a Banach space. Then, $$\sqrt{\kappa^2+\tau^2}\leq \mathcal{T}_1(\kappa,\tau,\mathcal{X})\leq\kappa +\tau .$$
\end{Proposition}

\begin{proof}
	To begin with, we aim to demonstrate that for any two-dimensional subspace  $\mathcal{Y}\subseteq\mathcal{X}$, there exist vectors $y_1,y_2\in \mathcal{S_Y}$ for which the equality $\bigg\|y_1+\frac{\tau}{\kappa}y_2\bigg\|=\bigg\|y_1-\frac{\tau}{\kappa}y_2\bigg\|$ holds. Drawing inspiration from the proof presented in \cite{05}, we consider the case where $\mathcal{Y}$ is equivalent to $\mathbb{R}^2$, treating it as a linear space endowed with a norm $\|\cdot\|$ and an orientation defined by the vector $y_3$. 
	
	Define $$\mathcal{C}(\kappa,\tau)=\bigg\{y_1\in \mathcal{Y}:\|y_1\| = \sqrt{1 + \frac{\tau^2}{\kappa^2}}\bigg\}$$and $$\mathcal{S}(\kappa,\tau)=\bigg\{y_1+\frac{\tau}{\kappa}y_2:~\bigg\|y_1 + \frac{\tau}{\kappa}y_2\bigg\|=\bigg\|y_1- \frac{\tau}{\kappa}y_2\bigg\|,y_1,y_2\in \mathcal{S_Y},[y_1,y_2]=y_3\bigg\}.$$
	Since \(\mathcal{A}(\mathcal{C}(\kappa,\tau))=(1 + \frac{\tau^2}{\kappa^2})\mathcal{A}(\mathcal{S_Y})\) and \(\mathcal{A}(\mathcal{S(\kappa,\tau)})=(1 + \frac{\tau^2}{\kappa^2})\mathcal{A(S_Y)}\) (where \(\mathcal{A(\cdot)}\) represents the area of the region bounded by the corresponding set), we have \(\mathcal{C}(\kappa,\tau)\cap \mathcal{S(\kappa,\tau)}\neq\varnothing\).
	Then, there exist $y_1,y_2\in \mathcal{S_Y}$ such that \(\|y_1 + \frac{\tau}{\kappa}y_2\|=\|y_1-\frac{\tau}{\kappa} y_2\|=\sqrt{1 + \frac{\tau^2}{\kappa^2}}\). 
	Thus, we get that $$\begin{aligned}
		\bigg(\|\kappa y_1+\tau y_2\|\|\kappa y_1-\tau y_2\|\bigg)^\frac{1}{2}&=\kappa\bigg(\bigg\|y_1 + \frac{\tau}{\kappa}y_2\bigg\|\bigg\|y_1- \frac{\tau}{\kappa}y_2\bigg\|\bigg)^\frac{1}{2}\\&=\sqrt{\kappa^2+\tau^2}.
	\end{aligned}$$
	It follows that that $ \mathcal{T}_1(\kappa,\tau,\mathcal{X})\geq\sqrt{\kappa^2+\tau^2},$ as desired.
	
	Conversely, since $$\begin{aligned}
		\bigg(\|\kappa x_1+\tau x_2\|\|\kappa x_1-\tau x_2\|\bigg)^\frac{1}{2}&\leq\frac{\|\kappa x_1+\tau x_2\|+\|\kappa x_1-\tau x_2\|}{2}\\&\leq\kappa+\tau.
	\end{aligned}$$
	Thus, we obtain that $\mathcal{T}_1(\kappa,\tau,\mathcal{X})\leq\kappa +\tau$.
\end{proof}
Subsequently, we give the estimation of the $\mathcal{T}_1(\kappa,\tau,\mathcal{X})$ in the $\ell_p$ space.
\begin{Example}\label{e4}
	Denote by $\ell_p$ the collection of all infinite sequences $\mathbf{x}=(x_n)_{n\in\mathbb{N}}$ of real or complex numbers satisfying $\sum_{n = 1}^{\infty}|x_n|^p<\infty$. For $\mathbf{x}\in\ell_p$, its $\ell_p$- norm is given by $\|\mathbf{x}\|_p=\left(\sum_{n = 1}^{\infty}|x_n|^p\right)^{\frac{1}{p}}$. Then for $p > 2$,   $\mathcal{T}_1(\kappa,\tau,\ell_p)=2^{-\frac{1}{p}}\big[(\kappa+\tau)^p+|\kappa - \tau|^p\big]^{\frac{1}{p}}$ holds.
\end{Example}
In order to prove Example \ref{e4}, we need to introduce several important inequalities as lemmas.
\begin{Lemma}\label{l3}\cite{07}
	Consider the function \(f(\kappa, \tau)=\kappa^{2}+\tau^{2}\) subject to the constraint \(g(\kappa, \tau)=\kappa^{p}+\tau^{p}-\mathcal{C} = 0\), where \(\mathcal{C}\) is a constant.
	If \(p\geq2\), then the function \(f(\kappa, \tau)\) attains its maximum value of \(2^{1 - \frac{2}{p}}\cdot (\kappa^p+\tau^p)^{\frac{2}{p}}\) at the point \((\kappa,\tau)\) where \(\kappa = \tau=(\frac{a}{2})^{\frac{1}{p}}\).
\end{Lemma}
\begin{Lemma}\label{l4}\cite{09}
	Let $\mathcal{X}$ be a Banach space, for any $x_1,x_2\in\mathcal{X}$, if $p>2$, then $$\|x_1+x_2\|^p+\|x_1-x_2\|^p\leq(\|x_1\|+\|x_2\|)^p+|\|x_1\|-\|x_2\||^p.$$
\end{Lemma}
\noindent\textbf{Proof of Example \ref{e4}}

First, by Lemma \ref{l4}, we have for any $x_1,x_2\in \mathcal{S}_{\ell_p}$,  \begin{equation}\label{e5}
	\|\kappa x_1+\tau x_2\|^p+\|\kappa x_1-\tau x_2\|^p\leq(\kappa+\tau)^p+|\kappa-\tau|^p.
\end{equation}
Then according to Lemma \ref{l3} and \eqref{e5}, we obtain that $$\begin{aligned}
	\|\kappa x_1+\tau x_2\|\|\kappa x_1-\tau x_2\|&\leq\frac{\|\kappa x_1+\tau x_2\|^2+\|\kappa x_1-\tau x_2\|^2}{2} \\&\leq\frac{2^{1-\frac{2}{p}}(\|\kappa x_1+\tau x_2\|^{p}+\|\kappa x_1-\tau x_2\|^{p})^{\frac{2}{p}}}{2}\\&\leq2^{-\frac{2}{p}}[(\kappa+\tau)^{p}+|\kappa-\tau|^{p}]^{\frac{2}{p}}.
\end{aligned}$$This implies that $\mathcal{T}_1(\kappa,\tau,\ell_p)\leq2^{-\frac{1}{p}}\big[(\kappa+\tau)^p+|\kappa-\tau|^p\big]^{\frac{1}{p}}.$

On the other hand, let $x_1=(\frac{1}{\sqrt[p]{2}},\frac{1}{\sqrt[p]{2}},0,0,\cdots)$ and $x_2=(\frac{1}{\sqrt[p]{2}},-\frac{1}{\sqrt[p]{2}},0,0,\cdots)$, then we have $x_1,x_2\in\mathcal{S}_{\ell_p}$ and $$	\|\kappa x_1+\tau x_2\|=\|\kappa x_1-\tau x_2\|=2^{-\frac{1}{p}}\big[(\kappa+\tau)^p+|\kappa-\tau|^p\big]^{\frac{1}{p}}.$$This implies that $\mathcal{T}_1(\kappa,\tau,\ell_p)\geq2^{-\frac{1}{p}}\big[(\kappa+\tau)^p+|\kappa-\tau|^p\big]^{\frac{1}{p}},$ as desired.

Next, we demonstrate the relationships between the $\mathcal{T}_1(\kappa,\tau,\mathcal{X})$ constant and other geometric constants as well as geometric properties in Banach spaces.

The following theorem  gives the relationship between $\mathcal{T}_1(\kappa,\tau,\mathcal{X})$ and $\mathcal{T}(\mathcal{X})$. 
\begin{Lemma}\cite{alonso2008geometric}\label{l5}
	Let $\mathcal{X}$ be a Banach space. For any $x_1 \in \mathcal{X}$
	$$r\in \mathbb{R}^+ \mapsto \phi_{x_1}(r) = \sup_{x_2 \in \mathcal{S(X)}} \|x_1 + rx_2\| \|x_1 - rx_2\|$$
	is an increasing function.
\end{Lemma}
\begin{Lemma}\label{l6}
	Let $\mathcal{X}$ be a Banach space. For any $x_1,x_2 \in \mathcal{S(X)}$, then the following conclusions hold:
	
	(i)~
	$t \in \mathbb{R}^+ \mapsto f(t)=\|x_1 + tx_2\| \|x_1 - tx_2\|$
	is an increasing function.	
	
	(ii)~$t \in \mathbb{R}^+ \mapsto g(t)=\|tx_1 + x_2\| \|tx_1 - x_2\|$
	is an increasing function.	
\end{Lemma}
\begin{proof}
	(i)~By Lemma \ref{l5}, it's obvious that this conclusion can be obtained.
	
	(ii)~From (i), we know that $x_1,x_2 \in \mathcal{S(X)}$ are arbitrary, thus we can obtain that $t \in \mathbb{R}^+ \mapsto g(t)=\|x_2 + tx_1\| \|x_2 - tx_1\|$
	is an increasing function, as desired.
\end{proof}
\begin{theorem}	
	Let $\mathcal{X}$ be a Banach space. Then 
	$$(\kappa\wedge\tau)\mathcal{T}(\mathcal{X})\leq\mathcal{T}_1(\kappa,\tau,\mathcal{X})\leq(\kappa\vee\tau)\mathcal{T}(\mathcal{X}).$$
\end{theorem}
\begin{proof}
	In order to establish the theorem, we analyze the problem by examining the subsequent three scenarios: 
	
	\textbf{Case(i)} If $\kappa=\tau$, then $$\begin{aligned}
		\mathcal{T}_1(\kappa,\tau,\mathcal{X})&=\sup\limits_{x_1,x_2\in \mathcal{S(X)}}\bigg(\|\kappa x_1+\tau x_2\|\|\kappa x_1-\tau x_2\|\bigg)^\frac{1}{2}\\&=\kappa\sup\limits_{x_1,x_2\in \mathcal{S(X)}}\bigg(\|x_1+x_2\|\|x_1-x_2\|\bigg)^\frac{1}{2}\\&=\tau\sup\limits_{x_1,x_2\in \mathcal{S(X)}}\bigg(\|x_1+x_2\|\|x_1-x_2\|\bigg)^\frac{1}{2}\\&=\kappa\mathcal{T}(\mathcal{X})=\tau\mathcal{T}(\mathcal{X}).
	\end{aligned}$$
	
	\textbf{Case(ii)}~If $\kappa<\tau$, we have $\frac{\kappa}{\tau}<1$ and $\frac{\tau}{\kappa}>1$ , then by Lemma \ref{l6}, we get that $f(\frac{\tau}{\kappa})>f(1)$ and $g(\frac{\kappa}{\tau})<g(1)$, which means that$$\bigg\|x_1 + \frac{\tau}{\kappa}x_2\bigg\| \bigg\|x_1 - \frac{\tau}{\kappa}x_2\bigg\|>\|x_1+x_2\|\|x_1-x_2\|$$ and $$\bigg\|\frac{\kappa}{\tau}x_1 + x_2\bigg\| \bigg\|\frac{\kappa}{\tau}x_1 - x_2\bigg\|<\|x_1+x_2\|\|x_1-x_2\|.$$ This implies that 
	$$\kappa\mathcal{T}(\mathcal{X})\leq\mathcal{T}_1(\kappa,\tau,\mathcal{X})\leq\tau\mathcal{T}(\mathcal{X}).$$
	
	\textbf{Case(iii)}~If $\kappa>\tau$, use the same method as in Case(ii), we can obtain that 	$$\tau\mathcal{T}(\mathcal{X})\leq\mathcal{T}_1(\kappa,\tau,\mathcal{X})\leq\kappa\mathcal{T}(\mathcal{X}).$$
	Combining the above three cases, we get that 	$$(\kappa\wedge\tau)\mathcal{T}(\mathcal{X})\leq\mathcal{T}_1(\kappa,\tau,\mathcal{X})\leq(\kappa\vee\tau)\mathcal{T}(\mathcal{X}).$$
\end{proof}

The following theorem aims to give the relationship between $\mathcal{T}_1(\kappa,\tau,\mathcal{X})$ and the uniformly non-squareness in Banach spaces, we give a lemma first.
\begin{Lemma}\cite{06}\label{l7}
	Let $\mathcal{X}$ be a  not uniformly non-square Banach space, then there exist $x_1,x_2\in\mathcal{S(X)}$, for any $\varepsilon>0$ such that $\|x_1+x_2\|\vee\|x_1-x_2\|>2-\frac{\varepsilon}{4}$.
\end{Lemma} 
\begin{theorem}
	
	Let $\mathcal{X}$ be a Banach space with $\mathcal{T}_1(\kappa,\tau,\mathcal{X})<\kappa +\tau$, then $\mathcal{X}$ is uniformly non-square.
	
\end{theorem}
\begin{proof}
	Without loss of generality, we start by assuming that $\kappa\geq\tau$. Given that $\mathcal{X}$ is not uniformly non-square, leveraging Lemma \ref{l7}, we obtain
	$$\begin{aligned}
		\|\kappa x_1+\tau x_2\|&\geq\kappa\|x_1+x_2\|-(\kappa-\tau)\\&\geq\kappa+\tau-\frac{\varepsilon}{4},
	\end{aligned}$$and$$\begin{aligned}
		\|\kappa x_1-\tau x_2\|&\geq\kappa\|x_1-x_2\|-(\kappa-\tau)\\&\geq\kappa+\tau-\frac{\varepsilon}{4}.
	\end{aligned}$$Let $\varepsilon\to 0$, we obtain that $\mathcal{T}_1(\kappa,\tau,\mathcal{X})\geq\kappa +\tau$.
\end{proof}
\vspace{-1.6cm}
\begin{figure}[!h]
	\centering
	\includegraphics[width=0.9\linewidth]{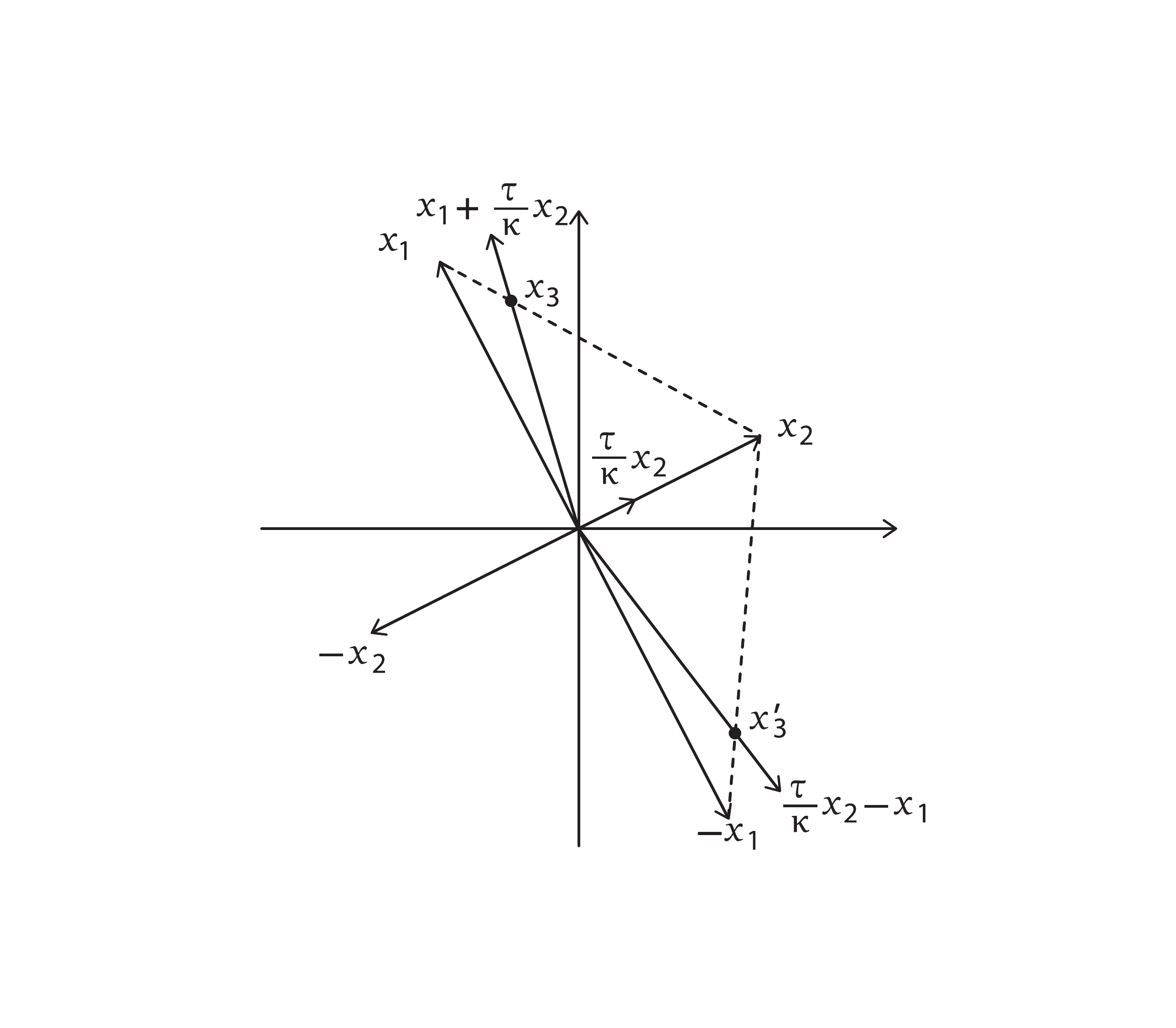}
	\caption{The vector diagram used to prove that $\mathcal{X}$ is uniformly non-square.}
	\label{f1}
\end{figure}

In order to describe the above theorem more accurately, we use the following geometric pictures to aid in comprehension. 

\begin{Remark}
	See the Figure \ref{f1}, we denote $x_3=[x_1, x_2]\cap[O, x_1+\frac{\tau}{\kappa}x_2]$ and since the triangle co$[\{x_1,x_3,x_1+\frac{\tau}{\kappa}x_2\}]$ is similar to the triangle co$[\{x_2,x_3,\frac{\tau}{\kappa}x_2\}]$, we have $\frac{\|x_1+\frac{\tau}{\kappa}x_2\|-\|x_3\|}{\|x_3\|}=\frac{\|\frac{\tau}{\kappa}x_2\|}{\|x_2
		\|}=\frac{\tau}{\kappa},$ thus, we get that $$\|x_1+\frac{\tau}{\kappa}x_2\|=(1+\frac{\tau}{\kappa})\|x_3\|\geq (1+\frac{\tau}{\kappa})\cdot\frac{\|x_1+x_2\|}{2}\geq(1+\frac{\tau}{\kappa})(1-\frac{\varepsilon}{8}).$$ Similarly, we have $$\|x_1-\frac{\tau}{\kappa}x_2\|=(1+\frac{\tau}{\kappa})\|x'_3\|\geq (1+\frac{\tau}{\kappa})\cdot\frac{\|x_1-x_2\|}{2}\geq(1+\frac{\tau}{\kappa})(1-\frac{\varepsilon}{8}).$$
	Then let $\varepsilon\to 0$, by the  equivalent definition  of the $\mathcal{T}_1(\kappa,\tau,\mathcal{X})$ constant, we obtain that $$\kappa\bigg(\|x_1+\frac{\tau}{\kappa}x_2\|\|x_1-\frac{\tau}{\kappa}x_2\|\bigg)^\frac{1}{2}\geq\kappa+\tau.$$This implies that $\mathcal{T}_1(\kappa,\tau,\mathcal{X})\geq\kappa+\tau.$ 
\end{Remark}

\section{Propoties of the skew constant $\mathcal{T}_2(\kappa,\tau,\mathcal{X})$}
By parameterizing $\mathcal{T}(\mathcal{X})$, we obtain another skew generalization of $\mathcal{T}(\mathcal{X})$, for papers on skew parameters, please refer to \cite{20,21,22,23}.  Now we present the key definition of the $\mathcal{T}_2(\kappa,\tau,\mathcal{X})$ constant:
\begin{Definition}
	For $\kappa,\tau>0$, we define $$\mathcal{T}_2(\kappa,\tau,\mathcal{X})=\sup\limits_{x_1,x_2\in \mathcal{S(X)}}\bigg(\|\kappa x_1+\tau x_2\|\|\tau x_1-\kappa x_2\|\bigg)^\frac{1}{2}.$$
\end{Definition}
\begin{Remark}
	$\mathcal{T}_2(1,1,\mathcal{X})=\mathcal{T}(\mathcal{X})=\sup\limits_{x_1,x_2\in \mathcal{S(X)}}\bigg(\|x_1+x_2\|\|x_1-x_2\|\bigg)^\frac{1}{2}.$
\end{Remark}
\begin{Lemma}\label{l1}
	If $\mathcal{H}$ is a Hilbert space, then for any $\kappa,\tau>0$, we have $$\mathcal{T}_2(\kappa,\tau,\mathcal{H})=\sqrt{\kappa^2+\tau^2}.$$
\end{Lemma}
\begin{proof}
	Assume that the norm of $\mathcal{H}$ is introduced by the inner product $\langle\cdot,\cdot\rangle$, then for any $\kappa,\tau>0$ and $x_1,x_2\in \mathcal{S(H)}$, we have $$\begin{aligned}
		&\bigg(\|\kappa x_1+\tau x_2\|\|\tau x_1-\kappa x_2\|\bigg)^\frac{1}{2}\\=&\bigg(\kappa^2\|x\|^2+\tau^2\|y\|^2+2\kappa\tau\langle x,y\rangle\bigg)^\frac{1}{2}\cdot\bigg(\kappa^2\|x\|^2+\tau^2\|y\|^2-2\kappa\tau\langle x,y\rangle\bigg)^\frac{1}{2}\\=&\bigg((\kappa^2+\tau^2)^2-(2\kappa\tau\langle x,y\rangle)^2\bigg)^\frac{1}{2},
	\end{aligned}$$thus, for a fixed $x_1\in \mathcal{S(H)}$,$$\begin{aligned}
		\sup_{x_2\in S_H}\bigg(\|\kappa x_1+\tau x_2\|\|\tau x_1-\kappa x_2\|\bigg)^\frac{1}{2}&=\sup_{x_2\in S_H}\bigg(\big((\kappa^2+\tau^2)^2-(2\kappa\tau\langle x,y\rangle)^2\big)^\frac{1}{2}\bigg)^\frac{1}{2}\\&=\sqrt{\kappa^2+\tau^2}, 
	\end{aligned}$$and the supremum is attained for $\langle x,y\rangle=0$.
	
	Hence, it follows that $\mathcal{T}_2(\kappa,\tau,\mathcal{H})=\sqrt{\kappa^2+\tau^2}$ by the arbitrariness of $x_1$.
\end{proof}

In the following Remark \ref{r4}, we only consider Banach space with $\operatorname{dim}\mathcal{X}\geq3$.
\begin{Remark}\label{r4}
	If $\mathcal{T}_2(1,1,\mathcal{X})=\sqrt{2}$,  then $\mathcal{X}$ is a Hilbert space.
\end{Remark}
\begin{proof}
	Since for any $x_1,x_2\in\mathcal{S(X)}$, $\|x_1+x_2\|\wedge\|x_1-x_2\|\leq\bigg(\|x_1+x_2\|\|x_1-x_2\|\bigg)^\frac{1}{2}$. Then, we take the supremum of both sides of the inequality, and we obtain $$\begin{aligned}
		\mathcal{J(X)}&=\sup\bigg\{\|x_1+x_2\|\wedge\|x_1-x_2\|:x_1,x_2\in\mathcal{S(X)}\bigg\}\\&\leq\sup\bigg\{\bigg(\|x_1+x_2\|\|x_1-x_2\|\bigg)^\frac{1}{2}:x_1,x_2\in\mathcal{S(X)}\bigg\}\\&=\mathcal{T}_2(1,1,\mathcal{X})=\sqrt{2}.
	\end{aligned}$$It follows that $\mathcal{J(X)}=\sqrt{2},$ which means that $\mathcal{X}$ is a Hilbert space \cite{12}.
\end{proof}
Use Lemma \ref{l1} and Dvoretzky's theorem \cite{02} (refer to Theorem 10.43 in the reference). Now we shall present the bounds of the $\mathcal{T}(\kappa,\tau,\mathcal{X})$ constant:
\begin{Proposition}\label{p1}
	Let $\mathcal{X}$ be a infinite Banach space. Then,$$\sqrt{\kappa^2+\tau^2}\leq\mathcal{T}_2(\kappa,\tau,\mathcal{X})\leq\kappa+\tau.$$
\end{Proposition}
\begin{proof}
	On the one hand, by Lemma \ref{l1} and Dvoretzki's theorem, we know that for any $\varepsilon> 0$, if the dimension of $\mathcal{X}$ is large enough, then there exists a subspace $\mathcal{Y}$ of $\mathcal{X}$ with dim$(\mathcal{Y}) = 2$ such that $\bigg|\mathcal{T}_2(\kappa,\tau,\mathcal{Y})-\sqrt{\kappa^2+\tau^2}\bigg|<\varepsilon$, which means that  $\mathcal{T}_2(\kappa,\tau,\mathcal{Y})\geq\sqrt{\kappa^2+\tau^2}$. Thus, we obtain that $\mathcal{T}_2(\kappa,\tau,\mathcal{X})\geq\sqrt{\kappa^2+\tau^2}$.
	
	Conversely, according to $$\begin{aligned}
		\sqrt{\|\kappa x_1+\tau x_2\|\|\tau x_1-\kappa x_2\|}&\leq\frac{\|\kappa x_1+\tau x_2\|+\|\tau x_1-\kappa x_2\|}{2}\\&\leq\kappa+\tau.
	\end{aligned}$$
	This implies that $\mathcal{T}_2(\kappa,\tau,\mathcal{X})\leq\kappa+\tau.$
\end{proof}

Now we will give an example to show that the upper bound of the $\mathcal{T}_2(\kappa,\tau,\mathcal{X})$ constant is sharp.
\begin{Example}
	Let $\mathcal{X}_1=(\mathcal{R}^{n},\|\cdot\|_1)$ where the norm $\|\cdot\|_1$ is defined by $$\|x_i\|_1=\|(x_{i1},x_{i2},\cdots)\|_1=\sum_{j=1}^{\infty}|x_{ij}|,$$ then $\mathcal{T}_2(\kappa,\tau,\mathcal{X}_1)=\kappa+\tau$.
\end{Example}
\begin{proof}
	Let $x_1=(1,0,0,\cdots)$ and $x_2=(0,1,0,\cdots)$, then by a simple calculation, we get that $\|\kappa x_1+\tau x_2\|_1=\kappa+\tau$ and $\|\tau x_1-\kappa x_2\|_1=\kappa+\tau$, thus, we obtain that $$ \sqrt{\|\kappa x_1+\tau x_2\|_1\|\tau x_1-\kappa x_2\|_1}=\kappa+\tau.$$ This implies that $\mathcal{T}_2(\kappa,\tau,\mathcal{X}_1)\geq\kappa+\tau$. 
	
	Hence, it follows that $\mathcal{T}_2(\kappa,\tau,\mathcal{X}_1)=\kappa+\tau$ according to  Proposition \ref{p1}.
\end{proof}

Next, we will show another example of the value of the $\mathcal{T}_2(\kappa,\tau,\mathcal{X})$ constant in a specific Banach space.

\begin{Example}
	Consider $\mathcal{X}_2$ be $\mathcal{R}^2$ with the $\ell_\infty-\ell_1$ norm, that is $$\|x_1\|=\|(x_{11},x_{12})\|=\left\{\begin{array}{ll}\|x_1\|_1=|x_{11}+x_{12}|,&x_1x_2\leq0,\\\|x_1\|_\infty=\max\{|x_{11}|,|x_{12}|\},&x_1x_2\geq0.\end{array}\right.$$Then, $\mathcal{T}_2(\kappa,\tau,\mathcal{X}_2)=(\kappa\vee\tau)(\kappa+\tau).$
\end{Example}
\begin{proof}
	First, we use extreme points to represent any $x_1$ and $x_2$ that belong to $\mathcal{S(X)}$: for any  $0\leq\varsigma\leq1$ and $0\leq\iota\leq1$, assume that $x_1=\varsigma a_{1}+(1-\varsigma)b_{1}$, and $x_2=\iota a_2+(1-\iota) b_2$, where $$a_1,a_2,b_1,b_2\in ext\{B_\mathcal{X}\}=\{(0,1), (1,0),(0,-1),(-1,0),(1,1),(-1,-1)\}.$$Subsequently,  substitute the aforesaid $a_i(i=1,2)$ and $b_i(i=1,2)$ into the equation below and consider  the convexity of $\|\cdot\|$, we get \begin{align*}	&\|\kappa x_1+\tau x_2\|\|\tau x_1-\kappa x_2\|\\=&\|\varsigma(\kappa a_{1}+\tau x_2)+(1-\varsigma)(\kappa b_1+\tau x_2)\|\|\varsigma(\tau a_{1}-\kappa x_2)\\&+(1-\varsigma)(\tau b_1-\kappa x_2)\|\\\leq&\big(\varsigma\|\iota(\kappa a_1+\tau a_2)+(1-\iota)(\kappa a_1+\tau b_2)\|\\&+(1-\varsigma)\|\iota(\kappa b_1+\tau a_2)+(1-\iota)(\kappa b_1+\tau b_2)\|\big)\\&\cdot\big(\varsigma\|\iota(\tau a_1-\kappa a_2)+(1-\iota)(\tau a_1-\kappa b_2)\|\\&+(1-\varsigma)\|\iota(\tau b_1-\kappa a_2)+(1-\iota)(\tau b_1-\kappa b_2)\|\big)\\\leq&\varsigma^2\iota^2\|\kappa a_1+\tau a_2\|\|\tau a_1-\kappa a_2\|\\&+\varsigma^2\iota(1-\iota)\|\kappa a_1+\tau a_2\|\|\tau a_1-\kappa b_2\|\\&+\varsigma\iota^2(1-\varsigma)\|\kappa a_1+\tau a_2\|\|\tau b_1-\kappa a_2\|\\&+\varsigma\iota(1-\varsigma)(1-\iota)\|\kappa a_1+\tau a_2\|\|\tau b_1-\kappa b_2\|\\&+\varsigma^2\iota(1-\iota)\|\kappa a_1+\tau b_2\|\|\tau a_1-\kappa a_2\|\\&+\varsigma^2(1-\iota)^2\|\kappa a_1+\tau b_2\|\|\tau a_1-\kappa b_2\|\\&+\varsigma\iota(1-\varsigma)(1-\iota)\|\kappa a_1+\tau b_2\|\|\tau b_1-\kappa a_2\|\\&+\varsigma(1-\varsigma)(1-\iota)^2\|\kappa a_1+\tau b_2\|\|\tau b_1-\kappa b_2\|\\&+\varsigma\iota^2(1-\varsigma)\|\kappa b_1+\tau a_2\|\|\tau a_1-\kappa a_2\|\\&+\iota(1-\varsigma)^2(1-\iota)\|\kappa b_1+\tau a_2\|\|\tau a_1-\kappa b_2\|\\&+(1-\varsigma)^2\iota^2\|\kappa b_1+\tau a_2\|\|\tau b_1-\kappa a_2\|\\&+\iota(1-\varsigma)^2(1-\iota)\|\kappa b_1+\tau a_2\|\|\tau b_1-\kappa b_2\|\\&+\varsigma\iota(1-\varsigma)(1-\iota)\|\kappa b_1+\tau b_2\|\|\tau a_1-\kappa a_2\|\\&+\varsigma(1-\varsigma)(1-\iota)^2\|\kappa b_1+\tau b_2\|\|\tau a_1-\kappa b_2\|\\&+\iota(1-\varsigma)^2(1-\iota)\|\kappa b_1+\tau b_2\|\|\tau b_1-\kappa a_2\|\\&+(1-\varsigma)^2(1-\iota)^2\|\kappa b_1+\tau b_2\|\|\tau b_1-\kappa b_2\|\\\leq&\max\big\{\|\kappa a_1+\tau a_2\|\|\tau a_1-\kappa a_2\|,\|\kappa a_1+\tau a_2\|\|\tau a_1-\kappa b_2\|,\\&\|\kappa a_1+\tau a_2\|\|\tau b_1-\kappa a_2\|,\|\kappa a_1+\tau a_2\|\|\tau b_1-\kappa b_2\|,\\&\|\kappa a_1+\tau b_2\|\|\tau a_1-\kappa a_2\|,\|\kappa a_1+\tau b_2\|\|\tau a_1-\kappa b_2\|,\\&\|\kappa a_1+\tau b_2\|\|\tau b_1-\kappa a_2\|,\|\kappa a_1+\tau b_2\|\|\tau b_1-\kappa b_2\|,\\&\|\kappa b_1+\tau a_2\|\|\tau a_1-\kappa a_2\|,\|\kappa b_1+\tau a_2\|\|\tau a_1-\kappa b_2\|,\\&\|\kappa b_1+\tau a_2\|\|\tau b_1-\kappa a_2\|,\|\kappa b_1+\tau a_2\|\|\tau b_1-\kappa b_2\|,\\&\|\kappa b_1+\tau b_2\|\|\tau a_1-\kappa a_2\|,\|\kappa b_1+\tau b_2\|\|\tau a_1-\kappa b_2\|,\\&\|\kappa b_1+\tau b_2\|\|\tau b_1-\kappa a_2\|,\|\kappa b_1+\tau b_2\|\|\tau b_1-\kappa b_2\|\big\}\\=&(\kappa\vee\tau)(\kappa+\tau).
	\end{align*}
	This implies that $\mathcal{T}_2(\kappa,\tau,\mathcal{X}_2)\leq(\kappa\vee\tau)(\kappa+\tau).$
	
	Conversely, let $x_1=(0,1),x_2=(1,0)$ we can easily get that $$\mathcal{T}_2(\kappa,\tau,\mathcal{X}_2)\geq(\kappa\vee\tau)(\kappa+\tau).$$
	Thus, we completed the proof.
\end{proof}
\begin{Definition}\cite{02}
	Let $(\mathcal{X}, \|\cdot\|)$ be a Banach space. The modulus of convexity $\delta_\mathcal{X}(\varepsilon)$  is defined as	
	$$
	\delta_\mathcal{X}(\varepsilon)=\inf\bigg\{1-\frac{\|x_1+x_2\|}{2}:\|x_1-x_2\|\geq\varepsilon, x_1,x_2\in \mathcal{S(X)}\bigg\},~0\leq\varepsilon\leq2.
	$$\end{Definition}Similarly, consider the relation with the $\mathcal{T}_2(\kappa,\tau,\mathcal{X})$ constant and the modulus of convexity $\delta_\mathcal{X}(\varepsilon)$, we have the following theorem:
\begin{theorem}\label{t3}
	Let $\mathcal{X}$ be a Banach space. Then for $0\leq\varepsilon\leq2$,$$\begin{aligned}
		&2(\kappa\vee\tau)^2\varepsilon\big(1-\delta_\mathcal{X}(\varepsilon)\big)-(\kappa\vee\tau)|\kappa-\tau|\big[2\big(1-\delta_\mathcal{X}(\varepsilon)\big)+\varepsilon\big]+(\kappa-\tau)^2\leq\mathcal{T}_2(\kappa,\tau,\mathcal{X})^2\\\leq&2(\kappa\wedge\tau)^2\varepsilon\big(1-\delta_\mathcal{X}(\varepsilon)\big)+(\kappa\wedge\tau)|\kappa-\tau|\big[2\big(1-\delta_\mathcal{X}(\varepsilon)\big)+\varepsilon\big]+(\kappa-\tau)^2.
	\end{aligned}$$
\end{theorem}
\begin{proof}
	On the one hand, since \begin{equation}\label{e6}
		\|\kappa x_1+\tau x_2\|\leq(\kappa\wedge\tau)\|x_1+x_2\|+|\kappa-\tau|,
	\end{equation} and\begin{equation}\label{e7}
		\|\tau x_1-\kappa x_2\|\leq(\kappa\wedge\tau)\|x_1-x_2\|+|\kappa-\tau|,
	\end{equation}
	then by the definition of the $\delta_{\mathcal X}(\varepsilon)$, we have $$\begin{aligned} 
		&\|\kappa x_1+\tau x_2\|\|\tau x_1-\kappa x_2\|\\\leq&\big((\kappa\wedge\tau)\|x_1+x_2\|+|\kappa-\tau|\big)\big((\kappa\wedge\tau)\|x_1-x_2\|+|\kappa-\tau|\big)\\\leq&\big(2(\kappa\wedge\tau)(1-\delta_{\mathcal X}(\varepsilon))+|\kappa-\tau|\big)((\kappa\wedge\tau)\varepsilon+|\kappa-\tau|)\\=&2(\kappa\wedge\tau)^2\varepsilon\big(1-\delta_\mathcal{X}(\varepsilon)\big)+(\kappa\wedge\tau)|\kappa-\tau|\big[2\big(1-\delta_\mathcal{X}(\varepsilon)\big)+\varepsilon\big]+(\kappa-\tau)^2.
	\end{aligned}$$
	It follows that $$\mathcal{T}_2(\kappa,\tau,\mathcal{X})^2\leq2(\kappa\wedge\tau)^2\varepsilon\big(1-\delta_\mathcal{X}(\varepsilon)\big)+(\kappa\wedge\tau)|\kappa-\tau|\big[2\big(1-\delta_\mathcal{X}(\varepsilon)\big)+\varepsilon\big]+(\kappa-\tau)^2.$$
	Conversely, let $\varepsilon\in[0,2]$, then for any $\eta>0$, there exists $x_1,x_2\in \mathcal{S(X)}$ such that$$\|x_1-x_2\|\geq\varepsilon,~~1-\frac{\|x_1+x_2\|}{2}\leq\delta_{\mathcal X}(\varepsilon)+\eta.$$
	Then, according to \begin{equation}\label{e8}
		\|\kappa x_1+\tau x_2\|\geq(\kappa\vee\tau)\|x_1+x_2\|-|\kappa-\tau|,
	\end{equation}
	and\begin{equation}\label{e9}
		\|\tau x_1-\kappa x_2\|\geq(\kappa\vee\tau)\|x_1-x_2\|-|\kappa-\tau|,
	\end{equation}
	we obtain that$$\begin{aligned}
		\mathcal{T}_2(\kappa,\tau,\mathcal{X})^2\geq& \|\kappa x_1+\tau x_2\|\|\tau x_1-\kappa x_2\|\\\geq&((\kappa\vee\tau)\|x_1+x_2\|-|\kappa-\tau|)((\kappa\vee\tau)\|x_1-x_2\|-|\kappa-\tau|)\\\geq&2(\kappa\vee\tau)^2\varepsilon\big(1-\delta_\mathcal{X}(\varepsilon)-\eta\big)\\&-(\kappa\vee\tau)|\kappa-\tau|\big[2\big(1-\delta_\mathcal{X}(\varepsilon)-\eta\big)+\varepsilon\big]+(\kappa-\tau)^2.
	\end{aligned}$$
	Since $\eta$ can be infinitely small, we let $\eta\to 0$, therefore, $$	\mathcal{T}_2(\kappa,\tau,\mathcal{X})^2\geq2(\kappa\vee\tau)^2\varepsilon\big(1-\delta_\mathcal{X}(\varepsilon)\big)-(\kappa\vee\tau)|\kappa-\tau|\big[2\big(1-\delta_\mathcal{X}(\varepsilon)\big)+\varepsilon\big]+(\kappa-\tau)^2.$$
\end{proof}

\begin{theorem}\label{t4}
	Let $\mathcal{X}$ be a Banach space. Then,$$\mathcal{T}_2(\kappa,\tau,\mathcal{X})^2\leq2\kappa^2\mathcal{C'_{NJ}(X)}+2\sqrt{2}\kappa|\kappa-\tau|\sqrt{\mathcal{C'_{NJ}(X)}}+(\kappa-\tau)^2.$$ 	
\end{theorem}
\begin{proof}
	Since	$$\begin{aligned}
		\|\kappa x_1+\tau x_2\|\|\tau x_1-\kappa x_2\|&\leq\frac{\|\kappa x_1+\tau x_2\|^2+\|\tau x_1-\kappa x_2\|^2}{2}\\&\leq\frac{(\kappa\|x_1+x_2\|+|\kappa-\tau|)^2+(\kappa\|x_1-x_2\|+|\kappa-\tau|)^2}{2}\\&=\frac{\kappa^2(\|x_1+x_2\|^2+\|x_1-x_2\|^2)+2\kappa|\kappa-\tau|(\|x_1+x_2\|+\|x_1-x_2\|)+2(\kappa-\tau)^2}{2}\\&\leq\frac{\kappa^2(\|x_1+x_2\|^2+\|x_1-x_2\|^2)+2\sqrt{2}\kappa|\kappa-\tau|\sqrt{\|x_1+x_2\|^2+\|x_1-x_2\|^2}+2(\kappa-\tau)^2}{2}.
	\end{aligned}$$
	This implies that $$\mathcal{T}_2(\kappa,\tau,\mathcal{X})^2\leq2\kappa^2\mathcal{C'_{NJ}(X)}+2\sqrt{2}\kappa|\kappa-\tau|\sqrt{\mathcal{C'_{NJ}(X)}}+(\kappa-\tau)^2.$$ 	
\end{proof}
\begin{Example}
	Let $\mathcal{X}$ be a not uniformly non-square space and $\kappa=2,\tau=3$. Then, $\mathcal{T}_2(2,3,\mathcal{X})=\sqrt{5}.$
\end{Example}
\begin{proof}
	If $\kappa=2,\tau=3$, then by Theorem \ref{t4}, we have$$\mathcal{T}_2(2,3,\mathcal{X})\leq\bigg(8\mathcal{C'_{NJ}(X)}+4\sqrt{2}\sqrt{\mathcal{C'_{NJ}(X)}}+1\bigg)^\frac{1}{2}.$$ 	Since $X$ is a not uniformly non-square space, it's known that $\mathcal{C'_{NJ}(X)}=2.$ Thus, we obtain that \begin{equation}\label{e10}
		\mathcal{T}_2(2,3,\mathcal{X})\leq\sqrt{5}.
	\end{equation}
	
	On the other hand, according to $X$ is a not uniformly non-square space, which means that there exist $x_1,x_2\in\mathcal{S(X)}$ such that$$\|x_1+x_2\|=2,\|x_1-x_2\|=2.$$Hence, we have $$5\geq\|2x_1+3x_2\|\geq3\|x_1+x_2\|-\|x_1\|=5,$$and$$5\geq\|3x_1-2x_2\|\geq3\|x_1-x_2\|-\|x_2\|=5.$$This implies that \begin{equation}\label{e11}
		\mathcal{T}_2(2,3,\mathcal{X})\geq\sqrt{5}.
	\end{equation}
	Combine \eqref{e10} and \eqref{e11}, we get that $\mathcal{T}_2(2,3,\mathcal{X})=\sqrt{5}.$
\end{proof}

Obviously, the $\mathcal{T}_2(\kappa,\tau,\mathcal{X})$ constant is a generalization of the $\mathcal{T}(\mathcal{X})$ constant, so it's necessary for us to study the relationship between them.

\begin{theorem}\label{t1}
	Let $\mathcal{X}$ be a Banach space. Then, $$\begin{aligned}
		&[(\kappa\vee\tau)\mathcal{T}(\mathcal{X})]^2-4(\kappa\vee\tau)|\kappa-\tau|+(\kappa-\tau)^2\\&\leq\mathcal{T}_2(\kappa,\tau,\mathcal{X})^2\leq[(\kappa\wedge\tau)\mathcal{T}(\mathcal{X})]^2+4(\kappa\wedge\tau)|\kappa-\tau|+(\kappa-\tau)^2.
	\end{aligned}$$
\end{theorem}
\begin{proof}
	By the above \eqref{e6} and \eqref{e7}, we obtain that$$\begin{aligned}
		&\|\kappa x_1+\tau x_2\|\|\tau x_1-\kappa x_2\|\\\leq&((\kappa\wedge\tau)\|x_1+x_2\|+|\kappa-\tau|)((\kappa\wedge\tau)\|x_1-x_2\|+|\kappa-\tau|)\\\leq&(\kappa\wedge\tau)^2\|x_1+x_2\|\|x_1-x_2\|+4(\kappa\wedge\tau)|\kappa-\tau|+(\kappa-\tau)^2.
	\end{aligned}$$
	This implies that $\mathcal{T}_2(\kappa,\tau,\mathcal{X})^2\leq[(\kappa\wedge\tau)\mathcal{T}(\mathcal{X})]^2+4(\kappa\wedge\tau)|\kappa-\tau|+(\kappa-\tau)^2.$
	
	Conversely, by the above \eqref{e8} and \eqref{e9}, we obtain that$$\begin{aligned}
		&\|\kappa x_1+\tau x_2\|\|\tau x_1-\kappa x_2\|\\\geq&\big((\kappa\vee\tau)\|x_1+x_2\|-|\kappa-\tau|\big)\big((\kappa\vee\tau)\|x_1-x_2\|-|\kappa-\tau\big)\\\geq&(\kappa\vee\tau)^2\|x_1+x_2\|\|x_1-x_2\|-4(\kappa\vee\tau)|\kappa-\tau|+(\kappa-\tau)^2,
	\end{aligned}$$
	it follows that $$\mathcal{T}_2(\kappa,\tau,\mathcal{X})^2\geq[(\kappa\vee\tau)\mathcal{T}(\mathcal{X})]^2-4(\kappa\vee\tau)|\kappa-\tau|+(\kappa-\tau)^2.$$
\end{proof}
In \cite{02}, the authors has proved that a Banach space $\mathcal{X}$ is not uniformly non-square is equavalent to $\mathcal{T}(\mathcal{X})=2$, from this, combine with Theorem \ref{t1}, we can obtain a simple proposition.
\begin{Proposition}\label{p10}
	A Banach space $\mathcal{X}$ is not uniformly non-square is equivalent to $\mathcal{T}_2(\kappa,\tau,\mathcal{X})=\kappa+\tau.$
\end{Proposition}
\begin{proof}
	Since the authors \cite{02} has proved that a Banach space $\mathcal{X}$ is not uniformly non-square is equavalent to $\mathcal{T}(\mathcal{X})=2$. Thus, we substitute $\mathcal{T}(\mathcal{X})=2$ into both sides of the inequality in Theorem \ref{t1}, and we obtain that $$[(\kappa\vee\tau)\mathcal{T}(\mathcal{X})]^2-4(\kappa\vee\tau)|\kappa-\tau|+(\kappa-\tau)^2=(\kappa+\tau)^2,$$and$$[(\kappa\wedge\tau)\mathcal{T}(\mathcal{X})]^2+4(\kappa\wedge\tau)|\kappa-\tau|+(\kappa-\tau)^2=(\kappa+\tau)^2.$$Hence, $\mathcal{T}_2(\kappa,\tau,\mathcal{X})=\kappa+\tau$ holds.
\end{proof}
\begin{Corollary}
	A Banach space $\mathcal{X}$ is not uniformly non-square is equivalent to any of the following conditions:
	
	(i)~$J(\mathcal{X})=2$;
	
	(ii)~$\mathcal{T}(\mathcal{X})=2$;
	
	(iii)~$\mathcal{T}_2(\kappa,\tau,\mathcal{X})=\kappa+\tau$.
\end{Corollary}

\begin{Definition}\cite{18}
	Let \(\mathcal{X}\) be a Banach space. For a closed, bounded, and convex subset \(\mathcal{K}\subseteq \mathcal{X}\), \(\mathcal{X}\) has normal structure if and only if 
	$r(\mathcal{K})<\mathrm{diam}(\mathcal{K}).$
	Furthermore, a Banach space \(\mathcal{X}\) is said to possess weak normal structure when, for any non - singleton weakly compact convex subset \(\mathcal{K}\) of \(\mathcal{X}\), \(r(\mathcal{K})<\mathrm{diam}(\mathcal{K})\) holds, where \(\mathrm{diam}(\mathcal{K})=\sup\limits_{k_1,k_2\in \mathcal{K}}\|k_1 - k_2\|\) and \(r(\mathcal{K})=\inf\limits_{k_1\in K}\sup\limits_{k_2\in K}\|k_1 - k_2\|\).
\end{Definition}

\begin{Lemma}\cite{17}\label{l8}
	Let $\mathcal{X}$ be a Banach space devoid of weak normal structure. Then $\forall\eta \in(0,1)$ and $\forall x_1 \in \mathcal{S(X)}$, $\exists x_2, x_3 \in \mathcal{S(X)}$ such that:
	
	(i)	$x_1=x_2-x_3$;
	
	(ii)$\left\|x_1+x_2 \right\|,\left\|x_3-x_1 \right\|>2(1-\eta)$.
	\\	where the six - point set $\left\{x_1, x_2, x_3,-x_1,-x_2,-x_3\right\} \subset \mathcal{S(X)}$ constitutes the vertices of an inscribed normal hexagon in $\mathcal{S(X)}$.
\end{Lemma}

\begin{theorem}
	Let $\mathcal{X}$ be a Banach space. Its normal structure is verified by the following three cases:
	
	(i)~ $0 < \tau<\kappa$ and $\mathcal{T}_2(\kappa,\tau,\mathcal{X})<\sqrt{\tau(\kappa + \tau)}$;
	
	(ii)~ $0<\kappa\leq\tau<2\kappa$ and $\mathcal{T}_2(\kappa,\tau,\mathcal{X})<\sqrt{\tau(3\kappa-\tau)}$;
	
	(iii)~ $\tau\geq2\kappa>0$ and $\mathcal{T}_2(\kappa,\tau,\mathcal{X})<\sqrt{(4\kappa-\tau)(3\kappa-\tau)}$.
\end{theorem}

\begin{proof}
	By Proposiyion \ref{p10}, we know that the situations (i), (ii), (iii) can all illustrate that $X$ is uniformly non-square, and thus reflexive \cite{11}. And it's known that if $\mathcal{X}$ is a reflexive Banach space, then $\mathcal{X}$ has weak normal structure is equivalent to $\mathcal{X}$ has normal structure. Thus, the normal structure in the theorem can be weakened to weak normal structure. Now we consider the situation where	$\mathcal{X}$ lacks weak normal structure. Then, by Lemma \ref{l8}, we have $$\begin{aligned}
		\|\kappa x_1+\tau x_2\|&=\|\kappa(x_1+x_2)+(\tau-\kappa)x_2\|\\&\geq\kappa\|x_1+x_2\|-|\kappa-\tau|
	\end{aligned}$$
	and$$\begin{aligned}
		\|\tau x_1-\kappa x_2\|&=\|\tau x_1-\kappa(x_1+x_3)\|\\&=\|\kappa(x_1-x_3)+(\tau-2\kappa)x_1\|
		\\&\geq\kappa\|x_1+x_2\|-|2\kappa-\tau|.
	\end{aligned}$$
	Thus, we have the following three cases:
	
	(i)If $0<\tau<\kappa$, then, we obtain that $$\begin{aligned}
		\|\kappa x_1+\tau x_2\|\geq2\kappa(1-\eta)-(\kappa-\tau)
	\end{aligned}$$
	and$$\begin{aligned}
		\|\tau x_1-\kappa x_2\|\geq2\kappa(1-\eta)-(2\kappa-\tau).
	\end{aligned}$$
	Hence, let $\eta\to0$, it follows that $\mathcal{T}_2(\kappa,\tau,\mathcal{X})\geq\sqrt{\tau(\kappa+\tau)},$ as desired.
	
	(ii)If $0<\kappa\leq\tau<2\kappa$, then, we obtain that $$\begin{aligned}
		\|\kappa x_1+\tau x_2\|\geq2\kappa(1-\eta)-(\tau-\kappa)
	\end{aligned}$$
	and$$\begin{aligned}
		\|\tau x_1-\kappa x_2\|\geq2\kappa(1-\eta)-(2\kappa-\tau).
	\end{aligned}$$
	Hence, let $\eta\to0$, we can conclude that $\mathcal{T}_2(\kappa,\tau,\mathcal{X})\geq\sqrt{\tau(3\kappa-\tau)}.$ 
	
	(iii)If $\tau \geq 2 \kappa>0$, we have $$\left\|\kappa x_1+\tau x_2\right\| \geq 2 \kappa(1-\eta)-(\tau-\kappa)$$ and $$\left\|\tau x_1-\kappa x_2\right\| \geq 2 \kappa(1-\eta)-(\tau-2 \kappa).$$ Letting $\eta \rightarrow 0$ yields $\mathcal{T}_2(\kappa, \tau, \mathcal{X}) \geq \sqrt{(3 \kappa-\tau)(4 \kappa-\tau)}$, as required.
\end{proof}

\section*{Data Availability Statement}
All type of data used for supporting the conclusions of this article is included in the article and also is cited at relevant places within the text as references.

\section*{Conflict of interest}
The authors declare that they have no conflict of interest.

\section*{ Funding Statement}
This work was supported by Anhui Province Higher Education Science
Research Project(Natural Science), 2023AH050487.

\end{document}